\documentclass[a4paper,11pt,leqno]{article}
\usepackage{amsfonts,amssymb,amsmath,amsthm}
\usepackage{epic,eepic}
\usepackage{mathrsfs}
\newtheorem{theorem}{Theorem}[section]
\newtheorem{proposition}[theorem]{Proposition}
\newtheorem{corollary}[theorem]{Corollary}
\newtheorem{lemma}[theorem]{Lemma}

\theoremstyle{remark}
\newtheorem{remark}[theorem]{{\bf Remark}}

\numberwithin{equation}{section}

\newcommand{\R}{{\mathbb R}}

\newcommand{\eps}{{\varepsilon}}
\newcommand{\vark}{{\varkappa}}

\renewcommand\a{\alpha}
\renewcommand\b{\beta}
\newcommand\g{\gamma}

\renewcommand\d{\delta}
\renewcommand\l{\lambda}
\newcommand\n{\nabla}
\renewcommand\o{\omega}\renewcommand\O{\Omega}

\def\sig{\sigma}
\def\vark{\varkappa}

\def\part{\partial}

\def\XXint#1#2#3{{\setbox0=\hbox{$#1{#2#3}{\int}$}
\vcenter{\hbox{$#2#3$}}\kern-.5\wd0}}

\def\W{\rlap{$\buildrel \circ \over W$}\phantom{W}}

\begin{document}

\title{
Estimates of  solutions for the parabolic $p$-Laplacian equation with measure
via parabolic nonlinear potentials }

\author{
{\large Vitali Liskevich}\\
\small Department of Mathematics\\
\small Swansea University\\
\small Swansea SA2 8PP, UK\\
{\small \tt v.a.liskevich@swan.ac.uk}\\
\and
{\large Igor I.\,Skrypnik}\\
\small Institute of Applied\\
\small Mathematics and Mechanics\\
\small Donetsk 83114,  Ukraine\\
{\small \tt iskrypnik@iamm.donbass.com} \and
{\large Zeev Sobol}\\
\small Department of Mathematics\\
\small  Swansea University\\
\small Swansea SA2 8PP, UK\\
{\small \tt z.sobol@swan.ac.uk} }

\date{}

\maketitle

\setlength{\unitlength}{0.0004in}
\begingroup\makeatletter\ifx\SetFigFont\undefined%
\gdef\SetFigFont#1#2#3#4#5{%
  \reset@font\fontsize{#1}{#2pt}%
  \fontfamily{#3}\fontseries{#4}\fontshape{#5}%
  \selectfont}%
\fi\endgroup%
\renewcommand{\dashlinestretch}{30}

\begin{abstract}
For 
weak solutions to the evolutional $p$-Laplace equation with a
time-dependent Radon measure on the right hand side we
obtain pointwise estimates via a nonlinear parabolic potential.
%
%
\end{abstract}

\bigskip

\section{Introduction and main results}
In this note we give a parabolic extension of a by now
classical result by Kilpel\"ainen-Mal\'y estimates~\cite{KiMa},
who proved pointwise estimates for solutions to quasi-linear
$p$-Laplace type elliptic equations with measure in the right
hand side, in terms of the (truncated) non-linear Wolff
potential $W^\mu_{\b,p} (x,R)$ of the measure,
\begin{equation}
\label{wolff}
W^\mu_{\b,p} (x,\rho)=\sum\limits_{j=0}^\infty
\left(\frac{\mu(B_{\rho_j}(x))}{\rho_j^{N-\b p}}\right)^\frac1{p-1},~ \rho_j:=2^{-j}\rho,~j=0,1,2,\ldots
\end{equation}
These estimates were subsequently extended to fully nonlinear
equations by Labutin \cite{Labutin} and fully nonlinear and
subelliptic quasi-linear equations by Trudinger and Wang
\cite{TW}. The pointwise estimates proved to be extremely
useful in various regularity and solvability problems for
quasilinear and fully nonlinear equations
\cite{KiMa,Labutin,PV1,PV2,TW}. For the parabolic equations the
corresponding result was recently given in \cite{DM2,DM3} for
the case $p=2$, and by the authors in \cite{LSS11} for
the case
$p>2$ and the measure on the right hand side depending on the spatial variable only.
One of the main difficulties in the time dependent measure case is that
of identifying the right analogue of the elliptic Wolff potential corresponding to $p$-Laplacian.

It is the aim of this note to introduce a parabolic version of
the Wolff potential and in terms of this newly defined
potential to establish pointwise estimates for solutions to
parabolic equations in the degenerate case $p\ge2$ with the
time-dependent measures on the right hand side. The form of the
parabolic potential introduced in the note is such that it
reduces to the truncated Wolff potential if the measure does
not depend on time, and  it reduces to the truncated Riesz
potential in the case $p=2$, so we recover the corresponding
result in \cite{DM2,DM3}.

We are concerned with
weak
solutions for
the divergence type quasi-linear parabolic
equations
\begin{equation}
\label{e0}
u_t- \Delta_p u = \mu\quad \text{in}\ \O_T:=\O\times (0,T),
\end{equation}
where $\O \subset \R^N$ is a domain and $T>0$, and $\mu$ is an
$\R^{N+1}$-valued (non-negative) Radon measure on $\O_T$. To
this end we introduce a parabolic analog of the non-linear
Wolff potentials.

Before formulating the main results, let us remind the reader
of the definition of a weak solution to equation \eqref{e0}.

We say that $u$ is a weak solution to \eqref{e0} if $u\in
V(\O_T):=C([0,T]; L^2_{loc}(\O))\cap L^p_{loc}(0,T;
W^{1,p}_{loc}(\O))$ and for any sub-domain $\O^\prime\Subset\O$
and any interval $I=[t_1,t_2]\subset (0,T)$ the integral
identity
\begin{equation}
\label{test}
\begin{split}
\int_{\O}u(t)\theta(t) dx\Big|_{t_1}^{t_2}+
\iint_{\O\times I}|\n u|^{p-2}\n u\n\theta\, dx\,dt\\
= \iint_{\O\times I} \theta\,d\mu + \iint_{\O\times I} u\partial_t\theta dx\,dt
\end{split}
\end{equation}
for any $\theta \in C_c^1(\O_T)$.

\bigskip

The crucial role in our
results is played by parabolic generalization of the truncated
Wolff potential, which is defined below.

\medskip

{\bf Parabolic Wolff potentials.} Let $\mu$ be a positive
measure on $\O_T$ and $(x_0,t_0)\in\O_T$. For $\rho,s>0$, let
$Q_{\rho,s}:=B_\rho(x_0)\times (t_0-s, t_0+s)$. For $\rho>0$
define
\begin{equation}\label{D}
    D_p(\rho):=\inf\limits_{\tau>0}\left\{i_p(\tau)+ \tfrac1{2(p-1)^{p-1}}\rho^{-N}\mu(Q_{\rho,\tau\rho^p})\right\},
\end{equation}
where
\begin{equation}\label{ip}
i_p(\tau):=
\begin{cases}
(p-2)\tau^{-\frac1{p-2}}, & p>2;
\\
\begin{cases}
+\infty, & \tau\in(0,1),\\
0, & \tau\ge1,
\end{cases}
& p=2.
\end{cases}
\end{equation}
Observe that $i_p(\tau)$ is continuous in $p$ for every $\tau>0$. Also
note that the above infimum is attained at some
$\tau\in(0,\infty]$ since the function under the infimum is
continuous in $\tau$. Moreover, $D_2(\rho)=\tfrac1{2}\rho^{-N}\mu(Q_{\rho,\rho^2}).$

Now let, for $\rho>0$ and for $j=0,1,2,\ldots$
set $\rho_j:=2^{-j}\rho$. We define the parabolic potential for a
measure $\mu$ as follows:
\begin{equation}\label{pWolff}
P_p^\mu(x_0,t_0;\rho):= \sum\limits_{j=0}^\infty D_p(\rho_j).
\end{equation}
In particular, there exists $\g>1$ such that
\begin{equation}
\label{P2}
\frac1\g P_2^\mu(x_0,t_0;r) \le \int\limits_0^r \rho^{-N}\mu(Q_{\rho,\rho^2})\frac{d\rho}\rho
\le\g P_2^\mu(x_0,t_0;r),
\end{equation}
so that for $p=2$ the introduced potential is equivalent to the truncated Riesz potential used in the estimates in~\cite{DM2,DM3}.
Note that, for a time-independent $\mu$ charging all balls
centered at $x_0$, the minimum in the definition of $D_p(\rho)$
is attained at
$\tau=(\tfrac1{(p-1)^{p-1}}\rho^{p-N}\mu(B_\rho))^{\frac{p-2}{p-1}}$.
So
\begin{equation}\label{autonom}
D_p(\rho)=\left[\rho^{p-N}\mu(B_\rho)\right]^{\frac1{p-1}},\quad
P_p^\mu(x_0,t_0;\rho) = W_p^\mu(x_0,\rho),
\end{equation}
so that in this case the introduced potential reduces to the non-linear Wolff potential.
Moreover, with $\tau(\rho)$ defined as follows:
\[
\tau(\rho)=\tau_\mu(\rho;x_0,t_0)
:=\left(\rho^{-N}\mu(Q_{\rho,\rho^p})\right)^{-\frac{p-2}{p-1}},
\]
it is easy to see that there exists $\g=\g_p>0$ such that, for
all $\rho>0$,
\[
D_p(\rho)\le \g\left(\rho^{-N}\mu(Q_{\rho,\rho^p})\right)^{\frac1{p-1}}
+ \g\rho^{-N}\mu(Q_{\rho,\tau(\rho)\rho^p})
\]
and that
\[
P^\mu_p(x_0,t_0;\rho)\le\g \sum\limits_{j=0}^\infty\left\{
\left(\rho_j^{-N}\mu(Q_{\rho_j,\rho_j^p})\right)^{\frac1{p-1}}
+ \rho_j^{-N}\mu(Q_{\rho_j,\tau(\rho_j)\rho_j^p})\right\}.
\]
Note that if $\mu$ is a time-independent measure then there
exists $\g>1$ such that
\[
\frac1\g W(x_0,\rho)\le \sum\limits_{j=0}^\infty\left\{
\left(\rho_j^{-N}\mu(Q_{\rho_j,\rho_j^p})\right)^{\frac1{p-1}}
+ \rho_j^{-N}\mu(Q_{\rho_j,\tau(\rho_j)\rho_j^p})\right\} \le \g W(x_0,\rho).
\]

\bigskip
The main result of this paper  is the following theorem.

\begin{theorem}
\label{mainthrm}Let $u$ be a weak solution to
equation~\eqref{e0}.
Then, for every $\l\in (0, \min \{\frac1{p-1},\frac1N \}]$
there exists $\g>0$ depending on $p,N,c_0,c_1$ and $\l$, such
that for every Lebesgue point $(y,s)\in \O_T$ of $u_\pm$ and
$\rho,\theta>0$ such that
$Q_{\rho,\theta}:=\{x:|x-y|\le\rho\}\times
[s-\theta,s+\theta]\subset \O_T$, with an additional assumption that
$\rho^2\le\theta$ in case $p=2$, one has
\[
\begin{split}
u_\pm(y,s)\le
\g\Bigg\{\eps_{\rho,\theta} + \left(\frac1{\rho^{N+p}}\iint_{Q_{\rho,\theta}} u_\pm^{(1+\l)(p-1)}dxdt\right)^\frac{1}{1+\l(p-1)}\\
+ P^{\mu_\pm}_p(y,s;\rho) 
\Bigg\},
\end{split}
\]
with
\[
\eps_{\rho,\theta}:=
\begin{cases}
\rho^{\frac p{p-2}}\theta^{-\frac1{p-2}}, & p>2,\\
0, & p=2.
\end{cases}
\]
\end{theorem}

The estimate above is not homogeneous in $u$ which is usual for
such type of equations \cite{DiB, DiGV}. 
The proof of Theorem~\ref{mainthrm} is based on a suitable
modifications of De Giorgi's iteration technique \cite{DG}
following the adaptation of Kilpel\"ainen-Mal\'y technique
\cite{KiMa} to parabolic equations with ideas from \cite{LS2,
Skr1}.

\begin{corollary}
Let $u$ be a weak solution to equation~\eqref{e0}. Assume that,
for all $\O'\Subset\O$ and $I\Subset (0,T)$,
\[
\lim\limits_{\rho\to0}\sup\limits_{(x,t)\in\O'\times I}P^{|\mu|}_p(x,t;\rho)<\infty.
\]
Then $u\in L^\infty_{loc}(\O_T).$
\end{corollary}

\begin{remark}
In case $\mu(dx,dt)=\mu(x,t)dxdt$ we can estimate $P_p^{|\mu|}$
by the Lebesgue and Lorentz norms as follows.
\begin{enumerate}
  \item Let $\mu\in L^r\big(0,T;L^q(\O)\big)$ for $r>1$ and
      $q>\frac Np$. Then
      \[
\rho^{-N}\mu(Q_{\rho,\rho^p\tau})\le \g\tau^{1-\frac1r}\rho^{p-\frac pr-\frac Nq}\|\mu\|_{q,r}
      \]
      and
      \[
D_p(\rho)\le \g \left[\rho^{p-\frac pr-\frac Nq}\|\mu\|_{q,r}\right]^{\frac1{p-1 - \frac1r(p-2)}}.
      \]
Hence, if $\frac1r+\frac N{pq}<1$ then
\[\sup\limits_{x,t,\rho}P_p^{|\mu|}(x,t;\rho)\le
  \g\|\mu\|_{q,r}^{\frac1{p-1 - \frac1r(p-2)}}.\] In particular, we
  recover a classical condition on local boundedness of the
  solution $u$ (see, e.g., \cite[Remark 0.1]{DiB86}).

By the same argument one proves that, for $\mu\in
  L^q\big(\O;L^r(0,T)\big)$ with $r>1$ and $q>\frac Np$
  such that $\frac1r+\frac N{pq}<1$, the following estimate
  holds:
\[\sup\limits_{x,t,\rho}P_p^{|\mu|}(x,t;\rho)\le
  \g\|\mu\|_{r,q}^{\frac1{p-1 - \frac1r(p-2)}}.\]

  \item The latter estimates can be refined in terms of the
      Lorentz norms. Recall that, for a measurable function
      $f$, the non-increasing rearrangement $f^*$ and its
      average $f^{**}$ are defined as follows:
      \[
f^*(s):=\inf\{t:\; |\{|f|(\xi)> t\}|\le t\}, \quad f^{**}(s):=\frac1s\int\limits_0^sf^*(\sig)d\sig
      \]
and that the spaces $L^{q,\a}$, $0<q,\a\le \infty$ are
defined by the following translation-invariant metrics:
\[
\|f\|_{q,\a}:=
\begin{cases}
\left[\int\limits_0^\infty \left(s^{\frac1q}f^{**}(s)\right)^\a\frac{ds}s\right]^{\frac1a},& 0<q,\a<\infty,\\
\sup\limits_{s>0}s^{\frac1q}f^{**}(s),& 0<q\le\infty,\ \a=\infty.
\end{cases}
\]
It is clear that
\[
\int_Ef(\xi)d\xi \le \int\limits_0^{|E|}f^*(s)ds = |E|f^{**}(|E|)\le |E|^{1-\frac1r}\|f\|_{r,\infty}.
\]
Let $\mu\in L^{q,\a}\big(\O;L^{r,\infty}(0,T)\big)$, with
      $r>\frac{p-2}{p-1}$, $q=\frac N{p-\frac pr}$ and
      $\a=\frac1{p-1 - \frac1r(p-2)}$. Then we estimate
      \[
      \begin{split}
\tfrac12\int\limits_{-\tau\rho^p}^{\tau\rho^p}\mu(x,t)dt\le (\tau\rho^p)^{1-\frac1r}\|\mu\|_{r,\infty}(x)
\mbox{ and }
\tfrac12\rho^{-N}\mu(Q_{\rho,\rho^p\tau})\\
\le \frac1{\o_N} \tau^{1-\frac1r}\rho^{p-\frac pr}\|\mu\|_{r,\infty}^{**}(\o_N\rho^N),
\end{split}
      \]
where $\o_N$ denotes the volume of a unit ball in $\R^N$.
Hence
\[D_p(\rho)\le \g \left[\rho^{p-\frac
pr}\|\mu\|_{r,\infty}^{**}(\o_N\rho^N)\right]^\a\] and
\[
\begin{split}
&\sup\limits_{x,t}P_p^{|\mu|}(x,t;\rho)\le \g\int\limits_0^\rho \left[s^{p-\frac
pr}\|\mu\|_{r,\infty}^{**}(\o_Ns^N)\right]^\a\frac{ds}s\\
&=\g\int\limits_0^{\o_N\rho^N}\left[s^{\frac{p-\frac
pr}N}\|\mu\|_{r,\infty}^{**}(s)\right]^\a\frac{ds}s
\le  \|\mu\|_{(r,\infty), (q,\a)}^\a.
\end{split}
\]
\end{enumerate}

\end{remark}

The rest of the paper contains the proof of Theorem
\ref{mainthrm}.

\section{ Proof of Theorem~\ref{mainthrm}}\label{bdd}

We start with some auxiliary integral estimates for the
solutions of \eqref{e0} which are formulated in the next lemma.
Let
\[
\eps_p:=\begin{cases}
(p-2)^{p-2}, & p>2;\\
1, & p=2.
\end{cases}
\]
Note that $\eps_p$ is continuous and that $\eps_p\ge e^{-\frac1e}>\frac12$. For $\l\in(0,1)$ we define
\begin{equation}\label{functions}
    G(s):=s_+^2\wedge s_+ \mbox{ and }
    \psi(s):=(1+s_+)^{1-\frac{1+\lambda}p} -1\asymp s_+\wedge s_+^{\frac{p-1-\lambda}p}.
\end{equation}
For $\d>0$ and $0<\rho<R$ define,
\[
I_\rho^{(\d)}(s):=(s-\eps_p\d^{2-p}\rho^p,\, s+\eps_p\d^{2-p}\rho^p),\quad Q_\rho^{(\d)}(y,s)=B_\rho(y)\times I_\rho^{(\d)}(s).
\]
In the sequel, $\g$ stands for a constant which depends only on
$N,p,c_0,c_1$ and $\l$, and which may vary from line to line.
\begin{lemma}
\label{lem2.3b} 
Let $\l\le\frac1{p-1}$  and $m\ge p$. Then there exists a
constant $\g>0$ depending only on $N,p,c_0,c_1,\l$ and $m$,
such that, for every solution $u$ to \eqref{e0} in $\O_T$,
every $l,\d>0$, and $(y,s)\in \O_T$ such that the cylinder
$Q_\rho^{(\d)}(y,s)\subset \O_T$, and every $\xi\in
C_c^\infty(Q_\rho^{(\d)}(y,s))$ such that $0\le\xi\le1$ and
$|\xi_t|\le 8\d^{p-2}\rho^{-p}$ and $|\n\xi|\le4\rho^{-1}$, the
following estimate holds.
\begin{equation}\label{main_est}
\begin{split}
&\sup\limits_{t\in I_\rho^{(\d)}(s)} \frac1{\rho^N}  \int_{B_\rho(y)}G\left(\frac{u-l}{\d}\right)\xi(x,t)^mdx\\
&+\frac{\d^{p-2}}{\rho^N}\iint_{L}
\left|\n\psi\left(\frac{u-l}\d\right)\right|^p\xi^mdx\,dt
\\
&\le  \g \frac{\d^{p-2}}{\rho^{p+N}}\iint_L G\left(\frac{u-l}\delta\right)\xi^{m-1}dx\,dt\\
& +\g \frac{\d^{p-2}}{\rho^{p+N}}\iint_L
\left(\frac{u-l}{\d}\right)^{(1+\l)(p-1)}\xi^{m-p}dx\,dt
+ \g \frac1{\d\rho^N} \mu_+\left(Q_\rho^{(\d)}(y,s)\right),
\end{split}
\end{equation}
where $L=Q_\rho^{(\d)}(y,s)\cap\{u>l\}$, $L(t)=L\cap
\{\tau=t\}$.
\end{lemma}
\begin{proof}
For shortness, we write $B:=B_\rho(y)$, $I:=I_\rho^{(\d)}(s)$
and $Q:=Q_\rho^{(\d)}(y,s)$. We also denote $I(t):=I\cap (0,t)$
and $Q(t):=B\times I(t)$.

Let
\begin{equation}\label{phi}
\begin{split}
    \phi(s):=\int\limits_0^{s_+}(1+\tau)^{-1-\lambda}d\tau \asymp s_+\wedge 1 \asymp \frac {s_+}{1+s_+}\\
\mbox{ and } \Phi(s):=\int\limits_0^s\phi(\tau)d\tau \asymp G(s)=s_+^2\wedge s_+.
\end{split}
\end{equation}
Let $m_\eps$ and $M_\sig$ denote symmetric mollifiers in $t$
and in $x$, respectively.
 Note that $m_\eps
M_\sig$ is a contraction in $L^q(Q)$ and $C\big(I;L^q(B)\big)$
for all $q\in[1,\infty]$ and that $m_\eps M_\sig \to I$ as
$\eps,\sig\to0$ in the strong operator topology of the
aforementioned spaces for $q\in[1,\infty)$. Also, $m_\eps
M_\sig\theta\to \theta$ a.e. on $Q$ as $\eps,\sig\to0$. Further
on, for a function $\theta$ we denote $\theta_{\eps}:=m_\eps
M_\eps\theta$.

We choose $\theta^{(\eps)}:=\frac1\delta \left[\phi
\left(\frac{u_{\eps}-l}\delta \right)\xi^m\right]_{\eps}$ as a
test function in \eqref{test}. Then we have that
\begin{equation}
\label{subst}
\begin{split}
\int_{B}u(t)\theta^{(\eps)}(t) dx +
\iint_{Q(t)} |\n u|^{p-2}(\n u)\n\theta^{(\eps)}dx\,dt\\
=\iint_{Q(t)} \theta^{(\eps)} d\mu +
\iint_{Q(t)} u\partial_t\theta^{(\eps)} dx\,dt.
\end{split}
\end{equation}
Note that $\theta^{(\eps)}\to\theta:=\frac1\delta \phi
\left(\frac{u-l}\delta \right)\xi^m$ in $C(I;L^q(B))\cap L^p(I;
\W^{1,p}(B))$ as $\eps\to0$ for all $q\in[1,\infty)$ since
$\phi$ is a bounded continuous function. Hence
\begin{equation}\label{limit-moll}
\begin{split}
\int_{B}u(t)\theta^{(\eps)}(t) dx + \iint_{Q(t)} |\n u|^{p-2}(\n u)\n\theta^{(\eps)}dx\,dt \\
{\rightarrow} \int_{B}u(t)\theta dx
+
\iint_{Q(t)} |\n u|^{p-2}(\n u)\n\theta dx\,dt \quad\text{as}\ \eps\to0.
\end{split}
\end{equation}
Since $m_\eps M_\sig$ is a contraction in $L^\infty(Q)$, we have that
$\theta^{(\eps)}\le \sup\phi=\frac1{\d\l}$. Therefore we obtain that
\begin{equation}\label{lower-order}
\iint_{Q(t)} \theta^{(\eps)}d\mu
\le \tfrac1{\d\l}\mu_+\big(Q(t)\big).
\end{equation}
Now we consider the last integral on the right hand side of
\eqref{subst}. Since $m_\eps M_\sig$ is a self-adjoint operator
commuting with the derivative,
\begin{equation*}
    \begin{split}
&\iint_{Q(t)}u\partial_t\theta^{(\eps)}dx\,dt =
\int_B u_\eps(t) \frac1\delta \phi \left(\frac{u_\eps(t)-l}\delta
\right)\xi^m(t) dx
\\
&- \iint\limits_{Q(t)} (\partial_t u_\eps) \frac1\delta \phi \left(\frac{u_\eps-l}\delta
\right)\xi^m dx\,dt
\\
= & \int_B u_\eps(t)\frac1\delta \phi \left(\frac{u_\eps(t)-l}\delta
\right)\xi^m(t) dx -\iint\limits_{Q(t)} \xi^m \partial_t\Phi \left(\frac{u_\eps-l}\delta
\right)dx\,dt
\\
= & \int_B u_\eps(t)\frac1\delta \phi \left(\frac{u_\eps(t)-l}\delta
\right)\xi^m(t) dx -\int\limits_{B} \Phi \left(\frac{u_\eps(t)-l}\delta
\right)\xi^m(t)dx
\\
& + m\iint\limits_{Q(t)} \Phi \left(\frac{u_\eps-l}\delta
\right)\xi^{m-1}\xi_tdx\,dt.
\end{split}
\end{equation*}
Since $\Phi$ is a Lipschitz continuous function, we conclude that
\begin{equation}\label{pbolic}
\begin{split}
\iint_{Q(t)}u\partial_t\theta^{(\eps)}dx\,dt
{\rightarrow}
\int_B u(t)\theta(t) dx
-\int\limits_{B} \Phi \left(\frac{u(t)-l}\delta
\right)\xi^m(t)dx\\
+ m\iint\limits_{Q} \Phi \left(\frac{u-l}\delta
\right)\xi^{m-1}\xi_tdx\,dt \quad\text{as} \ \eps\to0.
    \end{split}
\end{equation}

Collecting
\eqref{subst}--\eqref{pbolic} we obtain the following inequality:
\begin{equation*}
\begin{split}
&\int\limits_{B} \Phi \left(\frac{u(t)-l}\delta
\right)\xi^m(t)dx
+ \iint_{Q(t)} |\n u|^{p-2}(\n u)\n\theta dx\,dt
\\
&
\le m\iint\limits_{Q} \Phi \left(\frac{u-l}\delta
\right)\xi^{m-1}\xi_tdx\,dt + \frac1{\d\l}\mu_+\big(Q\big).
\end{split}
\end{equation*}
Taking the supremum in $t$, we obtain
\begin{equation}\label{sup-in-t}
    \begin{split}
&\sup\limits_{t\in I}\int\limits_{B} \Phi \left(\frac{u(t)-l}\delta
\right)\xi^m(t)dx
 + \iint_{Q} |\n u|^{p-2}\n u \n\theta dx\,dt
\\
&
\le m\iint\limits_{Q} \Phi \left(\frac{u-l}\delta
\right)\xi^{m-1}\xi_tdx\,dt + \frac1{\d\l}\mu_+\big(Q\big).
\end{split}
\end{equation}
Now we estimate the second term on the left hand side of
\eqref{sup-in-t} as follows.
\begin{equation}\label{elliptic}
\begin{split}
&\iint\limits_{Q} |\n u|^{p-2}\n u\n\theta
dx\,dt
\ge  \frac{1}{\d^{2}}\iint\limits_{L} \left(1+\frac{u-l}\delta
\right)^{-1-\lambda}|\n u|^p\xi^mdx\,dt
\\
& - \g\frac{1}{\l\d}\iint\limits_{L}|\n u|^{p-1}\left(1+\frac{u-l}\delta
\right)^{-1}\left(\frac{u-l}\delta\right)|\n \xi|\xi^{m-1}dx\,dt
\\
&\ge  \frac{1}{2\d^{2}} \iint\limits_{L} \left(1+\frac{u-l}\delta
\right)^{-1-\lambda}|\n u|^p\xi^mdx\,dt
\\
& - \g \frac{\d^{p-2}m^p}{\l^p}\iint\limits_{L} \left(1+\frac{u-l}\delta
\right)^{\lambda(p-1)-1}\left(\frac{u-l}\delta\right)^p|\n \xi|^p\xi^{m-p}dx\,dt.
\end{split}
\end{equation}
Observe now that $G \le \Phi\le \frac1\l G$ and
$\psi'(s)=(1+s)^{-\frac{1+\lambda}p}$, that $|\xi_t|\le
4\d^{p-2}\rho^{-p}$ and $|\n\xi|\le4\rho^{-1}$, and that
$(1+s)^{\l(p-1)-1}s^p\le s^{(1+\l)(p-1)}$ since $\l(p-1)\le 1$.
Hence we conclude from \eqref{pbolic} and \eqref{elliptic} that
\[
\begin{split}
&\sup\limits_{t\in I}\int_{L(t)}  G\left(\frac{u(t)-l}\delta\right)\xi^m(t)dx
+ \d^{p-2}\iint_L \left|\n \psi\left(\frac{u-l}\delta\right)\right|^p\xi^mdx\,dt
\\
&\le
\g\frac{\d^{p-2}}{\rho^p}\iint_L G\left(\frac{u-l}\delta\right)\xi^{m-1}dx\,dt
+\g \frac{\d^{p-2}}{\rho^p}\iint_L \left(\frac{u-l}\delta
\right)^{(1+\lambda)(p-1)}\xi^{m-p}dx\,dt
\\
&
 + \g \frac1{\d} \mu_+(Q).
\end{split}
\]
\end{proof}

\begin{remark}
The constant $\g$ in \eqref{main_est} is proportional to a
power of $m\max\phi=\frac m\l$, where $\phi$ is defined in
\eqref{phi}. In particular, it blows up as $\l\downarrow0$.
\end{remark}

\bigskip

Let $(y,s)$ be an arbitrary point in $\O_T$. Fix
$\rho,\theta>0$ such that
$\rho<\operatorname{dist}(y,\partial\O)$ and
$\theta<\min\{s,T-s\}$. For $p=2$ assume, in addition, that $\rho^2\le\theta$.
Fix $\d_{\rho,\theta}$:
\[
\d_{\rho,\theta}=
\begin{cases}
\left(\eps_p\rho^p\theta^{-1}\right)^{\frac1{p-2}}, & p>2,\\
0, & p=2.
\end{cases}
\]
Fix $m\ge 2p$ and $\xi\in C_c^\infty(B_1(0)\times (-1,1))$, such
that $0\le \xi\le 1$, $\xi(x,t)=1$ on
$B_{\frac12}(0)\times(-\frac12,\frac12)$, and $|\n \xi|<4$,
$|\partial_t\xi|<4$.

Fix a number $\varkappa\in (0,1)$
depending on $N,p,c_1,c_2$ and $\l$, which will be specified
later.

For $j=0,1,2,\dots$ positive numbers $l_j$ and $\d_j$ are
defined inductively as follows. We set
$\d_{-1}=2\d_{\rho,\theta}$ and $l_0=0$ and, for
$j=0,1,2,3,\ldots,$ given $\d_{j-1}$ and $l_{j}$, we define
$\d_j$ and $l_{j+1}$ as follows. We denote $\rho_j:=\rho 2^{-j}$, $B_j:=B_{\rho_j}(y)$ and
\[
\quad
\quad \tau_j:=
\sup\left\{\tau: i_p(\tau) + \tfrac1{2(p-1)^{p-1}}\rho^{-N}\mu(Q_{\rho_j,\tau\rho_j^p})=D_p(\rho_j)\right\},
\]
where $D(\rho_j)$ is as in \eqref{D}. For $\d\ge \hat\d_j$
with
\begin{equation}\label{min_delta}
\hat\d_j:=(\tfrac12\d_{j-1})\vee i_p(\tau_j),
\end{equation}
we define
\[
I_j^{\d}:=(s-\d^{2-p}\rho_j^p\eps_p,s+\d^{2-p}\rho_j^p\eps_p),
\ \ Q_j^{\d}:=B_j\times I_j^{\d},
\ \ L_j^\d:=\{(x,t)\in Q_j^\d:\,u(x,t)>l_{j}\}
\]
and, for $t\in I_j^\d$, \[L_j(t):=\{x\in B_j: u(x,t)>l_{j}\}.
\]
Then denote
\[
\xi_{j,\d}(x,t):=\xi\left(\frac{x-y}{\rho_j},\frac{t-s}{\d^{2-p}\rho_j^p\eps_p}\right).
\]
Note that $\xi_{j,\d}\in C_c^\infty(Q_j)$ and
$\xi_{j,\d}(x,t)=1$ for $(x,t)\in \frac12Q_j^\d$,
with the derivative estimates $|\n\xi_{j,\d}|\le 4
\rho_j^{-1}$, $|\partial_t \xi_{j,\d}|\le
4\d^{p-2}\rho_j^{-p}\eps_p^{-1}\le 8\d^{p-2}\rho_j^{-p} $.

Set
\begin{equation}\label{e3.1b}
\begin{split}
A_j(\d)=\frac{\d^{p-2}}{\eps_p\rho_j^{N+p}}\iint_{L_j^\d}\left(\frac{u-l_j}{\d}\right)^{(1+\l)(p-1)}\xi_{j,\d}^{m-p}dx\,dt\\
+\sup\limits_{t\in I_j^\d}\,
\frac1{\rho_j^N}\int_{L_j(t)}G\left(\frac{u-l_j}{\d}\right)\xi_{j,\d}^m
dx.
\end{split}
\end{equation}
For $j=0,1,2,\dots$,
if
\begin{equation}
\label{e3.3b}
A_j(\hat\d_j)\le \varkappa,
\end{equation}
we set $\d_j=\hat\d_j$ and $l_{j+1}=l_j+\d_j$.

Note that $A_j(\d)$ is continuous as a function of $\d$ and
$A_j(\d)\rightarrow 0$ as $\d\to \infty$. So if
\begin{equation}
\label{e3.4b}
A_j(\hat\d_j)> \varkappa,
\end{equation}
there exists $\hat\d>\hat\d_j$ such that
$A_j(\hat\d)=\varkappa$. In this case we set $\d_j=\hat\d$ and
$l_{j+1}=l_j+\d_j$.

With fixed $\d_j$, we set $I_j:=I_j^{\d_j}$, $Q_j:=Q_j^{\d_j}$,
$L_j:=L_j^{\d_j}$ and $\xi_j:=\xi_{j,\d_j}$.

The following proposition is a key in the Kilpel\"ainen-Mal\'y
technique \cite{KiMa}.
\begin{proposition}
\label{delta} One can choose $\vark>0$ such that there exists
$\g\ge1$ depending on the data, such that
\begin{equation}
\label{deltaj}
\d_j\le \tfrac12 \d_{j-1} + \g D_p(\rho_j),
\end{equation}
for $j= 1,2,3,\ldots,$ and, for $j=0$,
\begin{equation}\label{delta0}
\d_0\le \d_{\rho,\theta} + \g\Bigg(\frac1{\rho^{N+p}} \iint_{Q_{\rho,\theta}}u_+^{(1+\l)(p-1)}\Bigg)^{\frac1{1+\l(p-1)}}
+ \g D_p(\rho).
\end{equation}

\end{proposition}

The proof of Proposition \ref{delta} is split into several
lemmas.
\begin{lemma}\label{claim}
For $j= 1,2,3,\ldots,$ we have
\begin{equation}\label{imb1}
    Q_j\subset \tfrac12 Q_{j-1}~j=1,2,3\ldots,\mbox{ and } Q_j\subset Q_{\rho,\theta}~j=0,1,2,\ldots,
\end{equation}
so in particular $\xi_{j-1}\equiv1$ on $Q_j$, $j=1,2,3\ldots;$
\begin{equation}\label{imb2}
    Q_j\subset Q_{\rho_j,\tau_j\rho_j^p}
        ,~j=0,1,2,\ldots.
\end{equation}
\begin{equation}
\label{int0}
\frac{\d_j^{p-2}}{\eps_p \rho_j^{p+N}}|L_j|\le  \sup\limits_{t\in I_j}\frac{|L_j(t)|}{\rho_j^N}\le  2^N \vark
,~j=1,2,\ldots
\end{equation}
and
\begin{equation}\label{int1}
\sup\limits_{t\in I_j}\frac1{\rho_j^N}\int_{L_j(t)}\frac{u(x,t)-l_j}{\d_j}dx \le 2^{N+1}\vark
,~j=1,2,\ldots.
\end{equation}
\begin{equation}\label{int2}
\frac{\d_j^{p-2}}{\eps_p\rho_j^{N+p}}\int_{L_j}\left(\frac{u(x,t)-l_j}{\d_j}\right)^{(1+\l)(p-1)}dx\,dt
\le 2^{N+(1+\l)(p-1)}\vark
,~j=1,2,\ldots.
\end{equation}
There exists $\g>0$ such that, for $j=1,2,\ldots,$
\begin{equation}\label{est0}
\frac{\d_j^{p-2}}{\rho_j^N}\iint_{L_j}
\left|\n\psi\left(\frac{u-l_j}{\d_j}\right)\right|^p\xi^mdx\,dt
\le \g \vark
 + \g\frac1{\d_j\rho_j^N} \mu_+(Q_j).
\end{equation}

\end{lemma}

\begin{proof}
The imbedding \eqref{imb1}-\eqref{imb2} follows from the choice
$\d_j\ge \hat\d_{j}$, with $\hat\d_{j}$ defined in
\eqref{min_delta}. Indeed, since $\d_j\ge \frac12\d_{j-1}$, one
has $\d_j^{2-p}\rho_j^p \le \frac14\d_{j-1}^{2-p}\rho_{j-1}^p$.
Hence \eqref{imb1}. Similarly, $\d_j\ge i_p(\tau_j)$
implies $\eps_p\d_j^{2-p}\rho_j^p\le \tau_j\rho_j^p$.
Hence \eqref{imb2}.

To prove \eqref{int0}, observe that, for $(x,t)\in L_j$ one has
\begin{equation}
\label{Lj}
\frac{u(x,t)-l_{j-1}}{\d_{j-1}}=1+\frac{u(x,t)-l_{j}}{\d_{j-1}}\ge 1.
\end{equation}
Since $\xi_{j-1}=1$ on $Q_j$ and $I_j\subset I_{j-1}$ and
$L_j(t)\subset L_{j-1}(t)$ for $t\in I_j$, we obtain
\begin{equation}\label{proof_int0}
\begin{split}
\sup\limits_{t\in I_j}\frac{|L_j(t)|}{\rho_j^N}
&\le \rho_j^{-N}\sup\limits_{t\in I_j}\int_{L_j(t)}G\left(\frac{u-l_{j-1}}{\d_{j-1}}\right)\xi_{j-1}^mdx
\\
&\le 2^N \rho_{j-1}^{-N}\sup\limits_{t\in
I_{j-1}}\int_{L_{j-1}(t)}G\left(\frac{u-l_{j-1}}{\d_{j-1}}\right)\xi_{j-1}^mdx
\le 2^N\vark,
\end{split}
\end{equation}
which proves \eqref{int0}. To verify \eqref{int1}, note that
$G(s) + 1>s$ for $s\ge0$. Then, since $\d_j\ge\frac12
\d_{j-1}$, one has, for $(x,t)\in L_j$,
\begin{equation}\label{L1j}
\frac{u(x,t)-l_{j}}{\d_{j}}\le 2\frac{u(x,t)-l_{j}}{\d_{j-1}}= 2\frac{u(x,t)-l_{j-1}}{\d_{j-1}} - 2
\le 2 G\left(\frac{u(x,t)-l_{j-1}}{\d_{j-1}}\right).
\end{equation}
So, by the same argument as in \eqref{proof_int0},
\begin{equation*}
\begin{split}
&\sup\limits_{t\in I_j}\frac1{\rho_j^N}\int_{L_j(t)}\left(\frac{u(t)-l_{j}}{\d_{j}}\right)dx
\le \sup\limits_{t\in I_j}\frac2{\rho_j^N}\int_{L_j(t)}G\left(\frac{u(t)-l_{j-1}}{\d_{j-1}}\right)\xi_{j-1}^mdx
\\
&\le \sup\limits_{t\in I_{j-1}}
 \frac{2^{N+1}}{\rho_{j-1}^N}\int_{L_{j-1}(t)}G\left(\frac{u(t)-l_{j-1}}{\d_{j-1}}\right)\xi_{j-1}^mdx
\le 2^{N+1}\vark.
\end{split}
\end{equation*}

The estimate \eqref{int2} follows from the next observation:
\[
\begin{split}
\d_j^{p-2}\left(\frac{u-l_{j}}{\d_{j}}\right)^{(1+\l)(p-1)}
= & \left(\frac{\d_j}{\d_{j-1}}\right)^{p-2-(1+\l)(p-1)}\d_{j-1}^{p-2}\left(\frac{u-l_{j}}{\d_{j-1}}\right)^{(1+\l)(p-1)}
\\
\le &
2^{1+\l(p-1)}\d_{j-1}^{p-2}\left(\frac{u-l_{j-1}}{\d_{j-1}}\right)^{(1+\l)(p-1)}.
\end{split}
\]

To conclude \eqref{est0} from \eqref{main_est} one has to
estimate the first term in the right hand side of the latter.
To do this, it suffices to observe that $G(s)\le s$ and apply
\eqref{int1}.
\end{proof}

\begin{lemma}\label{eps-est}
For every $\eps>0$ there exist $\g_1(\eps), \g_2(\eps)>0$ such
that, for $j=1,2,3,\ldots,$
\begin{equation}\label{eps-est1}
\begin{split}
\frac{\d_j^{p-2}}{\rho_j^{N+p}}\iint_{L_j}  \left(\frac{u-l_j}{\d_j}\right)^{(1+\l)(p-1)}\xi_j^{m-p}dx\,dt
\\
\le  \eps\vark + \g_1(\eps)\vark^{\frac pN}\left(\vark +
 \frac1{\d_j\rho_j^N} \mu_+(Q_j)\right)
\end{split}
\end{equation}
and
\begin{equation}\label{eps-est2}
\begin{split}
&\sup\limits_{t\in I_j}\frac1{\rho_j^{N}}\int_{L_j(t)} G\left(\frac{u(t)-l_j}{\d_j}\right)  \xi_j^{m}dx
\\
\le & \eps\vark + \g_2(\eps)\vark^{\frac pN} \left(\vark +
 \frac1{\d_j\rho_j^N} \mu_+(Q_j)\right)
+\g \frac1{\d_j\rho_j^N} \mu_+(Q_j).
\end{split}
\end{equation}

\end{lemma}

\begin{proof}
For shortness we denote
\[w_j:=\psi\left(\frac{u-l_j}{\d_j}\right).\]
Note that, for every $\eps>0$, there exists $\g(\eps)>0$ such
that $s^{(1+\l)(p-1)}\le 2^{-N}\eps + \g(\eps)\psi^{p+\frac{\l
p^2}{p-1-\l}}(s)$. Hence, by \eqref{int0},
\begin{equation}\label{e3.10b}
\begin{split}
\frac{\d_j^{p-2}}{\rho_j^{N+p}}
\iint_{L_j}\left(\frac{u-l_j}{\d_j}\right)^{(1+\l)(p-1)}\xi_j^{m-p}dx\,dt\\
\le \eps\vark
+
\g(\eps)\frac{\d_j^{p-2}}{\rho_j^{N+p}}\iint_{L_j} w_j^{p+\frac{\l p^2}{p-1-\l}}\xi_j^{m-p}dx\,dt.
\end{split}
\end{equation}
The second term on the right hand side of \eqref{e3.10b} is
estimated by using the H\"older inequality first (note that
$\l\le\frac1N$), and then
the Sobolev inequality, as follows
\begin{equation}
\label{e3.11a}
\begin{split}
&\frac{\d_j^{p-2}}{\rho_j^{N+p}}  \iint_{L_j}  w_j^{p+\frac{\l p^2}{p-1-\l}}\xi_j^{m-p}dxdt\\
 &
\le \frac{\d_j^{p-2}}{\rho_j^{p}}
\int_{I_j}\left(\frac{|L_j(t)|}{\rho_j^N}\right)^{p(\frac 1N - \l)}\left(\frac1{\rho_j^N}\int_{L_j(t)}w_j^{\frac{ p}{p-1-\l}}dx\right)^{\l p}
\times\\
&\times\left(\frac1{\rho_j^N}\int_{L_j(t)}(w_j\xi_j)^{\frac{pN}{N-p}}\right)^{\frac{N-p}N}dt
\\
&
\le \g\,\left(\sup\limits_{t\in I_j}\frac{|L_j(t)|}{\rho_j^N} \right)^{p(\frac 1N - \l)}
\left(\sup\limits_{t\in I_j}\frac1{\rho_j^N}\int_{L_j(t)}w_j^{\frac{p}{p-1-\l}}dx\right)^{\l p}\times\\
&\times\left(\frac{\d_j^{p-2}}{\rho_j^{N}}\iint_{L_j}\left|\n \left(w_j \xi_j\right)\right|^p dx\,dt\right).
\end{split}
\end{equation}
Since $\psi(s)^{\frac{p}{p-1-\l}}\le \g s$ for $s\ge0$, the
first two factors in the right hand side of \eqref{e3.11a} are
estimated in \eqref{int0}-\eqref{int1} so that we obtain
\begin{equation*}
\begin{split}
\frac{\d_j^{p-2}}{\rho_j^{N+p}}\iint_{L_j} & \left(\frac{u-l_j}{\d_j}\right)^{(1+\l)(p-1)}\xi_j^{m-p}dx\,dt
\\
&\le \eps\vark + \g(\eps)\vark^\frac{p}N \,
\frac{\d_j^{p-2}}{\rho_j^{N}}\iint_{L_j}\left|\n \left(w_j
\xi_j\right)\right|^p dx\,dt\\
&\le \eps\vark +
\g(\eps)\vark^\frac{p}N \, \frac{\d_j^{p-2}}{\rho_j^N}\iint_{L_j}|\n w_j|^p \xi_j^p dx\,dt\\
+
&\g(\eps)\vark^\frac{p}N \, \frac{\d_j^{p-2}}{\rho_j^{N+p}}\iint_{L_j}w_j^p dx\,dt.
\end{split}
\end{equation*}
The second term on the right hand side of the last inequality is
estimated in \eqref{est0}. Then, the inequality $\psi^p(s)\le
\g(1+ s^{(1+\l)(p-1)})$ and \eqref{int0} and \eqref{int2} imply
that
\[
\frac{\d_j^{p-2}}{\rho_j^{N+p}}\iint_{L_j}w_j^p dx\,dt \le \g\vark.
\]
Hence \eqref{eps-est1} follows.

To conclude \eqref{eps-est2} from \eqref{main_est} and
\eqref{eps-est1}, we have to estimate the first term in the
right hand side of \eqref{main_est}. Note that, for every
$\eps>0$ there exists $\hat\g(\eps)>0$ such that $G(s)\le
2^{-N-1}\eps + \hat\g(\eps)s^{(1+\l)(p-1)}$. Then
\eqref{main_est} and \eqref{int0} imply that
\begin{equation*}
\begin{split}
&\sup\limits_{t\in I_j}\frac1{\rho_j^{N}}\int_{L_j(t)}  G\left(\frac{u(t)-l_j}{\d_j}\right)\xi_j^{m}dx
\\
\le & \tfrac12\eps\vark
+ (\g+\hat\g(\eps))
\frac{\d_j^{p-2}}{\rho_j^{N+p}}\iint_{L_j}\left(\frac{u-l_j}{\d_j}\right)^{(1+\l)(p-1)}\xi_j^{m-p}dx\,dt
+\g\frac1{\d_j\rho_j^N} \mu_+(Q_j),
\end{split}
\end{equation*}
with $\g>0$ as in \eqref{main_est}. Choose now $\eps_1>0$ such
that $\eps_1(\g+\hat\g(\eps))\le\frac12\eps$. Then
applying \eqref{eps-est1} with $\eps_1$ in place of $\eps$,
we obtain \eqref{eps-est2} with
$\g_2(\eps):=(\g+\hat\g(\eps))\g_1(\eps_1)$.
\end{proof}

\begin{proof}[Proof of Proposition \ref{delta}]
It suffices to prove \eqref{deltaj}-\eqref{delta0} in case
$\d_j>\hat\d_j.$ Otherwise the estimates are evident as
$\d_j=\hat\d_j$ implies that $\d_j=\frac12\d_{j-1}$ (recall
that $\frac12\d_{-1}=\d_{\rho,\theta}$) or
$\d_j=i_p(\tau_j)$. Note that $\d_j>\hat\d_j$
guarantees that $A_j(\d_j)=\varkappa$.

First we prove \eqref{deltaj}, that is, consider the case
$j=1,2,3,\ldots.$ Then it follows from Lemma \ref{delta} that,
for every $\eps>0$, there exists $\g(\eps)>0$ such that
\begin{equation}
\label{kappa}
\vark\le  \eps\vark + \g(\eps)\vark^{\frac pN}\left(\vark +
 \frac1{\d_j\rho_j^N} \mu_+(Q_j)\right) + \g\frac1{\d_j\rho_j^N} \mu_+(Q_j).
\end{equation}
Now choose $\eps=\frac12$ and $\vark$ such that
$\g(\frac12)\vark^{\frac pN}<\frac14$. Then it follows from
\eqref{kappa} that there exists $\g>0$ such that
\[
\frac1{\d_j\rho_j^N} \mu_+(Q_j)\ge \g\vark,\quad \mbox{hence}\quad \d_j\le \frac1{\g\vark}\frac1{\rho_j^N} \mu_+(Q_j).
\]
By \eqref{imb2}, $\mu_+(Q_j)\le
\mu_+(Q_{\rho_j,\tau_j\rho_j^p})$ so
\[
\d_j\le \tfrac12\d_{j-1} + i_p(\tau_j) +\g \rho_j^{-N}\mu_+(Q_{\rho_j,\tau_j\rho_j^p})
\le \tfrac12\d_{j-1} + \g D_p(\rho_j).
\]
So \eqref{deltaj} is shown.

Now we prove the estimate \eqref{delta0} of $\d_0$. Since
$A_0(\d_0)=\vark$, at least one of the following two
inequalities holds (recall that $l_0=0$):
\[
\tfrac12\vark\le \frac{\d_0^{p-2}}{\eps_p\rho^{N+p}}\iint_{Q_0}\left(\frac {u_+}{\d_0}\right)^{(1+\l)(p-1)}dx\,dt, \]
hence
\[
\d_0\le \left(\frac2{\vark\eps_p\rho^{N+p}}\iint_{Q_{\rho,\theta}}u_+^{(1+\l)(p-1)}dx\,dt\right)^{\frac1{1+\l(p-1)}},
\]
or
\begin{equation}\label{latter}
    \tfrac12\vark\le \sup\limits_{t\in I_0}\frac1{\rho^N}\int_{B_\rho}G\left(\frac {u_+}{\d_0}\right)\xi_0^mdx.
\end{equation}
In the former case \eqref{delta0} follows immediately, while in
the latter one we use \eqref{main_est} and the next estimate:
for every $\eps>0$ there exists $\g(\eps)>0$ such that $G(s)\le
\eps + \g(\eps)s^{(1+\l)(p-1)}$. Then \eqref{latter} implies
that, there exists $\g>0$ and, for every $\eps>0$ there exists
$\g(\eps)>0$ such that
\[
\tfrac12\vark\le  \g\eps
+ \g(\eps)\frac{\d_0^{p-2}}{\rho^{N+p}}\iint_{Q_0}\left(\frac {u_+}{\d_0}\right)^{(1+\l)(p-1)}dx\,dt
 + \g \frac1{\d_0\rho^N} \mu_+(Q_0).
\]
Choose $\eps>0$ such that $\g\eps\le \frac14\vark$. Then, for
some (other) $\g>0$,
\[
  \g\vark \le \frac{\d_0^{p-2}}{\rho^{N+p}}\iint_{Q_0}\left(\frac {u_+}{\d_0}\right)^{(1+\l)(p-1)}dx\,dt
    +  \frac1{\d_0\rho^N} \mu_+(Q_0).
\]
Thus at least one of the following two inequalities holds:
\[
\tfrac12 \g\vark \le \frac{\d_0^{p-2}}{\rho^{N+p}}\iint_{Q_0}\left(\frac {u_+}{\d_0}\right)^{(1+\l)(p-1)}dx\,dt,
\]
hence
\[
\quad \d_0\le \left(\frac2{\g\vark\rho^{N+p}}\iint_{Q_{\rho,\theta}}u_+^{(1+\l)(p-1)}dx\,dt\right)^{\frac1{1+\l(p-1)}},
\]
or
\[
\tfrac12 \g\vark\le \frac1{\d_0\rho^N} \mu_+(Q_0),
\quad\mbox{hence}\quad
\d_0\le \frac2{\g\vark} \frac1{\rho^N} \mu_+(Q_0).
\]
Note that $\mu_+(Q_0)\le \mu_+(Q_{\rho,\tau_0\rho^p})$, due to
\eqref{imb2}. Hence
\[
\begin{split}
\d_0\le & \d_{\rho,\theta} +
\g \left(\frac1{\rho^{N+p}}\iint_{Q_{\rho,\theta}}u_+^{(1+\l)(p-1)}dx\,dt\right)^{\frac1{1+\l(p-1)}}
+ i_p(\tau_0) + \g\mu_+(Q_{\rho,\tau_0\rho^p})
\\
\le & \d_{\rho,\theta} +
\g \left(\frac1{\rho^{N+p}}\iint_{Q(\rho)}u_+^{(1+\l)(p-1)}dx\,dt\right)^{\frac1{1+\l(p-1)}}
+\g D_p(\rho).
\end{split}
\]
So \eqref{delta0} holds.
\end{proof}

\begin{corollary}\label{l_infty}
The sequence $(l_j)$ is bounded above and
\[
l_j\nearrow l_\infty\le 2\d_{\rho,\theta} +
\g\left\{\left(\frac1{\rho^{N+p}}\iint_{Q_{\rho,\theta}} u_+^{(1+\l)(p-1)}dxdt\right)^\frac{1}{1+\l(p-1)}
+ P^{\mu_+}_p(y,s;\rho) 
\right\}.
\]
\end{corollary}
\begin{proof}
It follows from Proposition \ref{delta} and setting $l_0=0$
that there exists $\g>0$ such that, for $J=2,3,4,\ldots,$
\begin{equation*}
    \begin{split}
l_J& =\sum\limits_{j=0}^{J-1}\d_j \le  \d_0 + \tfrac12\sum\limits_{j=0}^{J-2}\d_j
 + \g\sum\limits_{j=1}^{J-1}D_p(\rho_j)
\\
\le & \tfrac12 l_{J-1} + \d_{\rho,\theta}
+ \g\left(\frac1{\rho^{N+p}}\iint_{Q_{\rho,\theta}} u_+^{(1+\l)(p-1)}dxdt\right)^\frac{1}{1+\l(p-1)}
+ \g\sum\limits_{j=0}^{J-1}D_p(\rho_j).
    \end{split}
\end{equation*}
Since $l_J>l_{J-1}$, the assertion follows.

\end{proof}

\begin{proof}[Proof of Theorem \ref{mainthrm}]
Since $\tilde u:=-u$ satisfies the equation $\partial_t \tilde
u - \Delta_p\tilde u = - \mu$, it suffices to show that
$u_+(y,s)\le l_\infty$ whenever $l_\infty<\infty$ and $(y,s)$
is a Lebesgue point for the function $u_+$.

Note that, by \eqref{imb1}, $Q_j\downarrow \{(y,s)\}$ as
$j\to\infty$. Observe that comparable symmetric cylinders form
a basis satisfying the Besicovitch property, by \cite[Lemma
1.6]{G70} (see also \cite[Chap.~I,\,Sec.1,\,Remark (5)]{G75}).
Hence, by \cite[Theorem 2.4]{G70} (see also
\cite[Chap.~II,\,Sec.~2,\, Theorem 2.1]{G75}), it is a
differentiable basis for all functions from $L^1_{loc}(\O_T)$.
So for a Lebesgue point $(y,s)$ for $u_+$, one has
\begin{equation*}
\begin{split}
u_+(y,s)= & \lim\limits_{j\to\infty}\frac1{|Q_j|}\iint_{Q_j}u_+dx\,dt
\le l_\infty + \limsup\limits_{j\to\infty}\frac1{|Q_j|}\iint_{Q_j}(u-l_\infty)_+dx\,dt
\\
\le & l_\infty +
\left(\limsup\limits_{j\to\infty}\frac1{|Q_j|}\iint_{Q_j}(u-l_\infty)_+^{(1+\l)(p-1)}dx\,dt\right)^{\frac1{(1+\l)(p-1)}}.
\end{split}
\end{equation*}
On the other hand
\begin{equation*}
\begin{split}
\frac1{|Q_j|}\iint_{Q_j}(u-l_\infty)_+^{(1+\l)(p-1)}dx\,dt<
\g\frac{\d_j^{p-2}}{\rho_j^{N+p}}\iint_{L_j}(u-l_j)^{(1+\l)(p-1)}dx\,dt\\
\le \g\vark \d_j^{(1+\l)(p-1)}\to0 \mbox{ as }j\to\infty,
\end{split}
\end{equation*}
since the series $\sum\d_j<\infty$. Hence the assertion
follows.
\end{proof}

%
%


\begin{small}

\end{small}


\end{document}